\documentclass[12pt, paperwidth=5.35in, paperheight=8.3in]{amsart}

\pdfpagewidth=\paperwidth
\pdfpageheight=\paperheight

\newtheorem{theorem}{Theorem}[section]
\newtheorem{lemma}[theorem]{Lemma}

\theoremstyle{definition}

\theoremstyle{remark}
\newtheorem{remark}[theorem]{Remark}
\newtheorem{corollary}[theorem]{Corollary}
\newtheorem{proposition}[theorem]{Proposition}
\numberwithin{equation}{section}

\allowdisplaybreaks
\usepackage{amsmath,amssymb,amscd,amsthm,amsxtra}
\usepackage[colorlinks=true, pdfstartview=FitV, linkcolor=blue, citecolor=blue, urlcolor=blue]{hyperref}
\usepackage{color}
\usepackage{graphicx,epstopdf,verbatim}

\def\BBR {{\mathbb R}}
\def\BBC {{\mathbb C}}
\newcommand{\R}{{\mathbb R}}

\newcommand {\miw}{  m_Iw }
\newcommand {\mjw}{  m_Jw }

\newcommand {\mlw}{  m_Lw }
\newcommand {\mkw}{  m_Kw }

\newcommand {\miwi}{ m_I w^{-1}}
\newcommand {\mjwi}{ m_J w^{-1}}

\newcommand {\mlwi}{ m_L w^{\scriptscriptstyle{-1}}}
\newcommand {\mkwi}{ m_K w^{\scriptscriptstyle{-1}}}

\newcommand {\mifw}{ m_I (fw) }

\newcommand {\wmiaf}{ m^{w}_I |f|}

\newcommand {\wimiag}{ m^{w^{\scriptscriptstyle{-1}}}_I |g|}

\newcommand {\miafw}{ m_I (|f|w)}

\newcommand {\diw}{\frac{|\Delta_I w|}{m_I w}}
\newcommand {\diwi}{\frac{|\Delta_I w^{-1}|}{m_I w^{\scriptscriptstyle{-1}}}}

\newcommand {\djwi}{\frac{|\Delta_J w^{-1}|}{m_J w^{\scriptscriptstyle{-1}}}}
\newcommand {\dkw}{\frac{|\Delta_K w|}{m_K w}}

\newcommand {\dkwi}{\frac{|\Delta_K w^{-1}|}{m_K w^{\scriptscriptstyle{-1}}}}

\newcommand {\diws}{\frac{|\Delta_I w|^2}{(m_I w)^2}}

\newcommand {\diwis}{\frac{|\Delta_I w^{-1}|^2}{(m_I w^{\scriptscriptstyle{-1}})^2}}

\newcommand{ \diwisp} {\frac{|\Delta_I w^{-\frac{1}{p-1}}|^2}{(m_I w^{\scriptscriptstyle{-\frac{1}{p-1}}})^2}}
\newcommand {\dkws}{\frac{|\Delta_K w|^2}{(m_K w)^2}}
\newcommand {\dkwis}{\frac{|\Delta_K w^{-1}|^2}{(m_K w^{\scriptscriptstyle{-1}})^2}}

\newcommand {\aI}{\alpha_I}

\newcommand {\bI}{\beta_I}

\newcommand {\wi}{w^{\scriptscriptstyle{-1}}}
\newcommand {\ap}{A^d_p}

\newcommand {\rhp}{RH^d_p}

\begin{document}
\title[DYADIC OPERATORS WITH COMPLEXITY $(m,n)$]{\textbf{WEIGHTED ESTIMATES FOR DYADIC PARAPRODUCTS AND $t$-HAAR MULTIPLIERS WITH COMPLEXITY $(m,n)$ }}

\author[J.C. MORAES]{JEAN CARLO MORAES} \thanks{The first author was supported by fellowship CAPES/FULBRIGHT, BEX 2918-06/4}
\address{Instituto de Matem\'atica, Universidade Federal do Rio Grande do Sul,  Av Bento Gon\c{c}alves 9500, 91501-970, Caixa Postal 15080, Porto Alegre, RS, Brazil}
\author[ M.C. Pereyra]{MAR\'{I}A CRISTINA PEREYRA}
\address{Department of Mathematics and Statistics, 1 University of New Mexico, Albuquerque, NM 87131-001, MSC01 1115}
\email{jean.moraes@ufrgs.br, crisp@math.unm.edu}
\subjclass[2010]{Primary 42C99 ; Secondary 47B38}
\keywords{Operator-weighted inequalities, Dyadic
paraproduct, $A_p$-weights, Haar multipliers.} \maketitle

\setcounter{equation}{0} \setcounter{theorem}{0}


\begin{abstract}
{We  extend the definitions of dyadic paraproduct and $t$-Haar multipliers
 to  dyadic operators that depend on the complexity $(m,n)$, for $m$ and $n$ natural numbers. We use the ideas
developed by Nazarov and Volberg to prove that the weighted
$L^2(w)$-norm of a paraproduct with complexity $(m,n)$, associated to a function $b\in BMO^d$, depends linearly on the $A^d_2$-characteristic of the weight $w$, linearly on the $BMO^d$-norm of $b$,
and polynomially on the complexity. This argument provides a new proof of the linear bound
 for the dyadic   paraproduct  due to Beznosova. 
We also prove that the $L^2$-norm of a $t$-Haar multiplier for any $t\in\mathbb{R}$ and weight $w$ is a multiple of the
square root of the $C^d_{2t}$-characteristic of $w$ times the square root
of the $A^d_2$-characteristic of $w^{2t}$, and is  polynomial  in the
complexity. }
\end{abstract}

\section{Introduction}

In the past decade, many mathematicians have devoted their attention to
finding out  how the norm of an operator $T$  on a weighted space $L^p(w)$ depends on
the so called $A_p$-characteristic of the weight $w$. More precisely, is there some optimal growth function $\varphi: [0,\infty)\to \BBR$
such that for all functions $f\in L^p(w)$,
$$ \|Tf\|_{L^p(w)} \leq C_{p,T}\varphi([w]_{A_p} )\|f\|_{L^p(w)},$$
where $C_{p,T}>0$ is a suitable constant?


The first result of this type was due to Buckley \cite{Bu} in 1993; he showed that  $\varphi (t)= t^{1/(p-1)}$ for the Hardy-Littlewood maximal function.
Starting in 2000, one at a time, some dyadic model operators 
and some important singular integral operators
(Beurling, Hilbert and Riesz transforms) were shown to obey a linear bound
with respect to the $A_2$-characteristic of $w $
in $L^2(w)$, meaning that for $p=2$, the function $\varphi (t) = t$ is the optimal one, see \cite{W,W1,HukTV,PetV,Pet2,Pet3,Be1}.
These linear estimates in $L^2(w)$  imply    $L^p(w)$-bounds  for $1<p<\infty$, by the sharp extrapolation
theorem  of Dragi\v{c}evi\v{c}, Grafakos,  Pereyra, and Petermichl,\cite{DGPPet}.
All these  papers used the Bellman function technique, 
see \cite{V} for more insights and references.

The linear bound for $H$, the Hilbert transform,  is based on a representation of $H$ as
an average of Haar shift operators of complexity $(0,1)$,  see \cite{Pet1}. Haar shift operators with complexity $(m,n)$
were introduced in \cite{LPetR}. Hyt\"onen obtained a representation valid for  \emph{any}
Calder\'on-Zygmund operator as an average of Haar shift operators of \emph{arbitrary complexity}, paraproducts and their adjoints,
and used this representation to
 prove the $A_2$-conjecture, see \cite{H}. Thus, he showed
 that for {\em all}  Calder\'on-Zygmund  operators $T$ in $\mathbb{R}^N$, and all weights $w\in A_p$,
 there is a constant $C_{p,N,T} >0$ such that,
 \[ \|Tf\|_{L^p(w)}\leq C_{p,N,T} [w]_{A_p}^{\max\{1,{1}/{p-1}\}}\|f\|_{L^p(w)}.\]


 See \cite{L1} for a survey of the $A_2$-conjecture including a rather  complete
history of most results that  appeared up to November 2010, and that contributed to the final resolution
of this mathematical puzzle.
A crucial part of the proof was to obtain bounds for Haar shifts operators that depended
linearly on the $A_2$-characteristic and at most polynomially on the
complexity $(m,n)$.  In 2011, Nazarov and Volberg \cite{NV}  provided a beautiful new proof that still
uses Bellman functions, although minimally, and that can be transferred to geometric doubling metric spaces \cite{NV1, NRezV}.
Treil \cite{T}, independently   \cite{HLM+}   obtained  linear dependence on the complexity. Similar Bellman function techniques have been used to
prove the Bump Conjecture in $L^2$, see  \cite{NRezTV}.  

 It seems natural to study
  other dyadic operators  with complexity $(m,n)$, and examine if we can recover the same
dependence on the $A_2$-characteristic that we have for the original
operator (the one with complexity $(0,0)$) times a factor that
depends at most polynomially on the complexity of these operators.
We will do this analysis  for the dyadic paraproduct
and for the $t$-Haar multipliers.


 For $b\in BMO^d$, a function of dyadic bounded mean oscillation,  $m,n\in \mathbb{N}$,
the {\em dyadic paraproduct of complexity $(m,n)$} is defined by,
\[ \pi_b^{m,n}f(x)= \sum_{L\in\mathcal{D}}\sum_{\substack{I\in\mathcal{D}_n(L)\\ J\in\mathcal{D}_m(L)}} c^L_{I,J}m_If\, \langle b,h_I\rangle h_J(x), \]
where  $ |c^L_{I,J}| \leq {\sqrt{|I|\,|J|}}/{|L|}$, and $m_If$ is the average of $f$ on the interval $I$.
Here $\mathcal{D}$ denotes the dyadic intervals, $|I|$  the length of interval $I$,
 $\mathcal{D}_m(L)$ denotes the dyadic subintervals of $L$  of length $2^{-m}|L|$,
  $h_I$ are the Haar functions, and $\langle f,g\rangle$ denotes the $L^2$-inner product on $\BBR$.

 We prove that  the dyadic paraproduct of complexity $(m,n)$  obeys the same \emph{linear bound}  as obtained by Beznosova
\cite{Be1} for the dyadic paraproduct of complexity $(0,0)$ (see \cite{Ch1} for the result in $\BBR^N$, $N>1$),
multiplied by a  factor that depends polynomially on the complexity.

\begin{theorem}\label{thm:paracplxty(m,n)} If  $w\in A^d_2$, $b\in BMO^d$,  then
\[ \|\pi_b^{m,n}f\|_{L^2(w)}\leq C(m+n+2)^5[w]_{A^d_2}\|b\|_{BMO^d}\|f\|_{L^2(w)}.\]
\end{theorem}

  Our proof of Theorem~\ref{thm:paracplxty(m,n)}  shows how to use the ideas in \cite{NV} for this setting, explicitly displaying the dependence on $\|b\|_{BMO^d}$ and
 bypassing the more complicated Sawyer  two-weight testing conditions present in other arguments \cite{HPzTV,L1,HLM+}. From our point view, this makes the proof
 more transparent.


For $t\in\mathbb{R}$, $m,n\in \mathbb{N}$, and  weight $w$,  the {\em $t$-Haar multiplier
of complexity $(m,n)$} is defined  by
\[ T_{t,w}^{m,n}f(x)= \sum_{L\in\mathcal{D}} \sum_{\substack{I\in\mathcal{D}_n(L)\\ J\in\mathcal{D}_m(L)}} c^L_{I,J}\, \frac{w^t(x)}{(m_Lw)^t}\langle f,h_I\rangle h_J(x),\]
where $|c^L_{I,J}|\leq {\sqrt{|I|\, |J|}}/{|L|}$.
When $(m,n)=(0,0)$  and $c^L_{I,J}=c^L=1$ for all $L\in\mathcal{D}$,
 we denote the corresponding Haar multiplier by $T^t_w$. In addition, if $t=1$, we denote the multiplier simply by $T_w$. A necessary condition for the boundedness
of $T_w^t$ on $L^2(\BBR )$ is that $w\in C^d_{2t}$, that is,
\[[w]_{C^d_{2t}}:=\sup_{I\in\mathcal{D}} \Big (\frac{1}{|I|}\int_I w^{2t}(x)dx \Big )\Big (\frac{1}{|I|}\int_I w(x)dx\Big )^{-2t}<\infty.\]
This condition is also sufficient for $t<0$ and $t > 1/2$.
For $0\leq t \leq 1/2$ the condition $C^d_{2t}$ is always fulfilled; in this case, boundedness of $T^t_w$ is known when $w\in A^d_{\infty}$, see \cite{KP}.
The Haar multipliers $T_w$  are closely related to the resolvent of the dyadic paraproduct \cite{P}, and appeared
in the study of Sobolev spaces on Lipschitz curves \cite{P3}.
It was proved in \cite{P2} that
the $L^2$-norm for the Haar multiplier $T_w$ depends linearly on the
$C_2^d$-characteristic of the weight $w$.
  We   show the following theorem that generalizes a result of Beznosova for $T_w^t$ \cite[Chapter 5]{Be}.
\begin{theorem}
If $w\in C^d_{2t}$  and  $w^{2t}\in A_2^d$, then
\[ \|T_{t,w}^{m,n}f\|_2\leq C (m+n+2)^3[w]_{C_{2t}}^{\frac{1}{2}}[w^{2t}]^{\frac{1}{2}}_{A^d_2}\|f\|_2.\]
The condition $w\in C^d_{2t}$ is necessary for the boundedness of $T_{t,w}^{m,n}$ when
$c^L_{I,J}= \sqrt{|I| |J|}/|L|$.
\end{theorem}
The result is optimal for $T^{\pm 1/2}_w$, see \cite{Be, P2} and \cite{BeMoP}.  We expect that, for both the paraproducts
and $t$-Haar multipliers with complexity $(m,n)$, the dependence on the complexity can be strengthened to be linear, in line
with the best results for the Haar shift operators. However our methods yield polynomials of degree 5 and 3 respectively.




To simplify notation, and to shorten the exposition we analyze  the one-dimensional case.
Some of the building blocks in our arguments
can be found in the literature in the case of $\R^N$, or even in the geometric doubling metric space case.
As we go along we will note where such results can be found. For a complete presentation
of these results in the geometric doubling metric spaces (in particular in $\R^N$) see \cite{Mo1}.

The paper is organized as follows. In Section 2 we provide the basic definitions and results that are
used throughout this paper. In  Section 3 we  prove the lemmas that are essential for  the main results.
In Section 4 we  prove the main estimate for the dyadic paraproduct with complexity $(m,n)$ and
 present a new proof of the linear bound for the dyadic paraproduct.
In Section 5 we  prove the main estimate for the $t$-Haar multipliers with complexity $(m,n)$, also discussing
necessary conditions for these operators to be bounded in $L^p(\mathbb{R})$, for $1 < p < \infty$.

\vskip .1in
\noindent {\bf Acknowledgements:} The authors would like to thank
Carlos P\'erez,
Rafael Espinola and Carmen Ortiz-Caraballo for organizing the Doc-
course: \emph{Harmonic analysis, metric spaces and applications to PDE}, held
in Seville, at the Instituto de Matem\'aticas de la Universidad de Sevilla
(IMUS)  during the Summer of 2011. We are grateful  to our thoughtful
referees who  pointed out multiple ways for
improving this paper.


\section{Preliminaries}\label{preliminaries}

\subsection{Weights, maximal function  and dyadic intervals}
 A weight $w$ is a locally integrable function in
$\mathbb{R}^N$  taking values in $(0,\infty)$ almost everywhere.
The $w$-measure of a measurable set $E$, denoted by $w(E)$, is
$ w(E)= \int_E w(x)dx.$
For a measure $\sigma$, $ \sigma(E) = \int_{E} d\sigma$, and $|E|$
stands for the Lebesgue measure of $E$. We define $m^{\sigma}_E f$ to be the
integral average of $f$ on $E$, with respect to $\sigma$,
$$m^{\sigma}_E f := \frac{1}{\sigma(E)} \int_E f(x)
d\sigma.$$
When $d\sigma=dx$ we simply write $m_Ef$; when $d\sigma = v\,dx$ we write $m_E^v f$.

Given a weight $w$, a measurable function $f: \BBR^N\to\BBC$ is
 in $L^p(w)$ if and only if $\|f\|_{L^p(w)}:= \left (\int_{\mathbb{R}} |f(x)|^pw(x)dx \right )^{1/p}<\infty$.

For a weight $v$ we define the {\em weighted maximal function of $f$} by
$$ (M_vf)(x) := \sup_{Q \ni x} m_Q^v |f|,$$
where $Q$ is a cube in $\BBR^N$ with sides parallel to the axes.
The operator  $M_v$ is bounded in $L^q(v)$ for all $q>1$.
Furthermore,
\begin{equation}\label{bddmaxfct}
\|M_v f\|_{L^q(v)} \leq C_N q' \|f\|_{L^q(v)},
\end{equation}
\noindent where  $q'$ is the dual exponent of
$q$, that is ${1}/{q} + {1}/{q'}=1$. A
proof of this fact can be found in \cite{CrMPz1}. When $v=1$, $M_v$ is the usual
Hardy-Littlewood maximal function, which we will denote by $M$.
It is well-known that $M$ is bounded on $L^p(w)$ if and only if $w\in A_p$ \cite{Mu}.

We  work with the collection of all  {\em dyadic intervals}, $\mathcal{D}$, given by:
$\mathcal{D}= \cup_{n\in\mathbb{Z}}\mathcal{D}_n$,
 where
$\mathcal{D}_n:=\{I \subset \mathbb{R} : I=[k2^{-n}, (k+1)2^{-n} ), \; k \in \mathbb{Z}\}.$
For a dyadic interval $L$ , let $\mathcal{D}(L)$ be  the collection of its dyadic subintervals,
$\mathcal{D}(L):=\{I \subset L :  I \in \mathcal{D}\} ,$
 and let $\mathcal{D}_n(L)$ be the $n^{th}$-\emph{generation} of dyadic subintervals of $L$,
$\mathcal{D}_n(L):=\{I \in \mathcal{D}(L) : |I|=2^{-n}|L| \}.$
Any two dyadic intervals $I, J \in \mathcal{D}$ are either disjoint
or one is contained in the other. Any two distinct dyadic intervals $I,
J \in \mathcal{D}_n$ are disjoint, furthermore $\mathcal{D}_n$ is a partition of $\mathbb{R}$, and
$\mathcal{D}_n(L)$ is a partition of $L$. For every dyadic interval $I \in
\mathcal{D}_n$ there is exactly one $\widehat{I} \in
\mathcal{D}_{n-1}$, such that $I  \subset \widehat {I}$;
$\widehat{I}$ is called the \emph{parent of $I$}. 
Each dyadic interval $I$ in $\mathcal{D}_n$  is the union of two disjoint
intervals in $\mathcal{D}_{n+1}$, the right and left halves, denoted $I_+$ and $I_-$ respectively,
and called the \emph{children} of $I$.

A weight $w$ is {\em dyadic doubling} if  ${w(\widehat{I})}/{w(I)} \leq C\;$ for all $\; I \in \mathcal{D}$.
The smallest constant $C$ is called the doubling constant of $w$ and
 is denoted by $D(w)$. Note that $D(w)\geq 2$, and that in fact the ratio between the length
of a child and the length of its parent is comparable to one; more precisely,
$ D(w)^{-1}\leq {w(I)}/{w(\widehat{I})}\leq 1- D(w)^{-1}$.

\subsection{Dyadic $A^d_p$, reverse H\"older $RH_p^d$ and $C_s^d$ classes}
A  weight $w$ is said to belong to the {\em  dyadic Muckenhoupt $A_p^d$-class} if and only if
\[  [w]_{A_p^d}:= \sup_{I\in \mathcal{D}} (\miw )(m_Iw^{\frac{-1}{p-1}})^{p-1} <\infty,\quad\quad\mbox{for}\quad 1< p<\infty ,\]
where $[w]_{A_p^d}$ is called the $A_p^d$-characteristic of the weight.
If a weight is in $A_p^d$  then it is dyadic doubling. These classes are nested:
$A_p^d\subset A_q^d$ for all $p\leq q$.
The class $A^d_{\infty}$ is defined by $A^d_{\infty}:= \bigcup_{p>1}A_p^d$.

A  weight $w$ is said to belong to the {\em dyadic reverse H\"older  $RH_p^d$-class} if and only if
\[ [w]_{RH_p^d}:= \sup_{I\in \mathcal{D}}(m_Iw^p)^{\frac{1}{p}}(m_Iw)^{-1}<\infty, \quad \quad\mbox{for}\quad 1<p<\infty,\]
where $[w]_{RH_p^d}$ is called the $RH_p^d$-characteristic of the weight.
If a weight is in $RH_p^d$  then it is not necessarily dyadic doubling
(in the non-dyadic setting reverse H\"older weights are always doubling). Also these classes are nested,
$RH_p^d\subset RH_q^d$ for all $p\geq q$.
The class $RH^d_1$ is defined by $RH^d_1:= \bigcup_{p>1}RH_p^d$.
 In the non-dyadic setting  $A_{\infty}=RH_1$.
 In the dyadic setting  the collection of dyadic doubling weights in $RH_1^d$ is $A_{\infty}^d$,  hence $A_{\infty}^d$ is
 a proper subset of $RH_1^d$.
See \cite{BeRez} for some recent and very interesting results relating these classes.

 Given  $s \in\mathbb{R}$, a weight $w$ is said to satisfy the {\em $C^d_s$-condition} if
$$[w]_{C^d_s}:= \sup_{I \in \mathcal{D}} \big (m_I w^s\big )\,\big (\miw\big )^{-s} < \infty.$$
The quantity defined above is called the $C^d_s$-characteristic of
$w$. The class  $C^d_s$ was defined in \cite{KP}. Let us analyze this definition.
For $0 \leq s \leq 1$, we have that any weight satisfies the
condition with $C_s^d$-characteristic $1$, being just a
consequence of H\"older's Inequality (cases $s=0,1$ are trivial).
When $s>1$, the condition is equivalent to the dyadic reverse
H\"older condition and
$[w]^{{1}/{s}}_{C^d_s} =[w]_{RH^d_s}.$
For $s<0$, we have that  $ w \in C^d_s$ if and only if $w \in A^d_{1- 1/s}.$ Moreover
$[w]_{C^d_s} = [w]^{-s}_{A^d_{1-1/s}}$.

\subsection{Weighted Haar functions}
For a given weight $v$ and an interval $I$  define the corresponding {\em weighted Haar function} by
\begin{equation}\label{def:Haarfunction}
 h^v_I(x)= \frac{1}{v(I)}\left (  \sqrt{\frac{v(I_-)}{v(I_+)}}\,\chi_{I_+}(x)-\sqrt{\frac{v(I_+)}{v(I_-)}} \,\chi_{I_-}(x)\right ),
 \end{equation}
where $\chi_I$  is the characteristic function of the interval $I$.

If $v$ is the Lebesgue measure on $\mathbb{R}$, we will denote the
{\em Haar function} simply by $h_I$.
It is an important fact that $\{h^v_I\}_{I\in \mathcal{D}}$ is an orthonormal system
in $L^2(v)$, with the inner product $\langle f, g\rangle_v = \int_{\mathbb{R}} f(x)\,\overline{g(x)}\,v(x) dx$.

It is a simple exercise to verify that the weighted and unweighted Haar functions are related linearly as follows:
\begin{proposition}\label{whaarbasis}
For any weight $v$, there are numbers $\alpha_I^v$,
$\beta^v_I$ such that
$$ h_I(x) = \alpha^v_I \,h^v_I(x) + \beta_I^v \,{\chi_I(x)}/{\sqrt{|I|}}$$
where
{(i)} $|\alpha^v_I | \leq \sqrt{m_Iv},$
{ (ii)}  $|\beta^v_I| \leq {|\Delta_I v|}/{m_Iv},$
 $\Delta_I v:= m_{I_+}v - m_{I_-}v.$
\end{proposition}

For a  weight $v$ and a dyadic interval $I$,
${|\Delta_I v|}/{m_I v}
=2 \Big| 1 - {m_{I_-}v}/{m_{I}v}  \Big| \leq 2.$
If the weight $v$ is dyadic doubling then we get an improvement on the above
upper bound,
${|\Delta_I v|}/{m_I v} \leq 2\left (1-{2}/{D(v)}\right ).$ 


\subsection{Dyadic BMO and Carleson sequences}
A locally integrable function $b$ is a function of {\em  dyadic bounded mean
oscillation}, $b \in BMO^d$, if and only if
\begin{equation}\label{def:BMO}
\|b\|_{BMO^d}:= \Big (\sup_{J \in \mathcal{D}} \frac{1}{|J|} \sum_{I
\in \mathcal{D}(J)} | \langle b , h_I \rangle|^2
\Big )^{\frac{1}{2}}< \infty.
\end{equation}
Note that if $\displaystyle{b_I := \langle b , h_I
\rangle }$ then $|b_I|\,|I|^{-\frac{1}{2}} \leq
\|b\|_{BMO^d}$, for all $ \; I \; \in \mathcal{D}$.

If $v$ is a weight,  a positive sequence $\{\alpha_I\}_{I\in \mathcal{D}}$ is called
a {\em $v$-Carleson sequence with intensity $B$} if for all $J\in \mathcal{D}$,
\begin{equation}\label{def:vCarlesonseq}
({1}/{|J|}) \sum_{I \in \mathcal{D}(J)}
{\lambda_I}\leq B\; m_Jv.
\end{equation}
When $v=1$ we call a sequence satisfying \eqref{def:vCarlesonseq}  for all ${J \in \mathcal{D}} $ a
{\em Carleson sequence with intensity} $B$.  If $b
\in BMO^d$ then $\{ |b_I |^2\}_{I\in \mathcal{D}} $ is a Carleson
sequence with intensity $\|b\|^2_{BMO^d} $.

\begin{proposition}\label{algcarseq}
Let $v$ be a weight, 
$\{\lambda_I\}_{I\in \mathcal{D}}$ and $\{\gamma_I\}_{I\in \mathcal{D}}$ be two $v$-Carleson sequences with
intensities $A$ and $B$ respectively then for any $c, d >0$ we have
that
\begin{itemize}
\item [(i)]$ \{c \lambda_I + d\gamma_I\}_{I\in \mathcal{D}}$ is a $v$-Carleson sequence with
intensity  $cA + dB$.

\item [ (ii)]$\{ \sqrt{\lambda_I} \sqrt{\gamma_I}\}_{I\in \mathcal{D}}$ is a $v$-Carleson sequence
with intensity  $\sqrt{AB}$.

\item [(iii)]$ \{( c\sqrt{\lambda_I} + d \sqrt{\gamma_I})^{2}\}_{I\in \mathcal{D}}$ is a $v$-Carleson sequence
with intensity  $2c^2A+2d^2B$.

\end{itemize}

\end{proposition}

The proof of these statements is quite simple. To prove the first
one we just need properties of  the supremum, for the second one we
 apply the Cauchy-Schwarz inequality, and the third one is a consequence of the
first two statements combined with the fact that
$2cd\sqrt{A}\sqrt{B} \leq  c^2A +  d^2B. $

\section{Main tools}

In this section, we state and prove the lemmas and theorems necessary
to obtain the  estimates for the paraproduct and the $t$-Haar multipliers
of complexity $(m,n)$.
The Weighted Carleson  Lemma~\ref{weightedCarlesonLem},  $\alpha$-Lemma~\ref{alphacoro} and
Lift Lemma~\ref{liftlem} are fundamental for all our estimates.

\subsection{Carleson Lemmas}

We present some  weighted Carleson lemmas that we will use.
Lemma~\ref{folklem} was introduced and used in \cite{NV}, it was called a folklore lemma in reference to the likelihood
of having been known before. Here we obtain Lemma~\ref{folklem}
as an immediate corollary of the Weighted Carleson Lemma~\ref{weightedCarlesonLem}  and what
 we call the Little Lemma~\ref{litlem}, introduced by Beznosova in her proof of the linear bound for the dyadic paraproduct.

\subsubsection{Weighted Carleson Lemma}
The Weighted Carleson Lemma we present here is a variation in the spirit of other weighted Carleson embedding
theorems that appeared before in the literature \cite{NV, NTV1}. All the lemmas in this section hold in $\R^N$ or
even geometric doubling metric spaces, see \cite{Ch1,NRezV}.

\begin{lemma}[Weighted Carleson Lemma]\label{weightedCarlesonLem}
Let $v$ be a dyadic doubling
 weight, then $\{\alpha_{L}\}_{L \in \mathcal{D}}$ is a $v$-Carleson sequence
with intensity $B$ if and only if for all   non-negative $v$-measurable functions $F$ on the
line,
\begin{equation}\label{eqn:WCL}
\sum_{L \in \mathcal{D}}\alpha_{L} \inf_{x \in L} F(x) \leq B
\int_{\mathbb{R}}F(x) \,v(x)\,dx.
\end{equation}
\end{lemma}

\begin{proof}

\noindent($\Rightarrow$) Assume that $F \in  L^1(v)$ otherwise the first statement
is automatically true. Setting $\displaystyle{\gamma_L =
\inf_{x \in L}F(x)}$, we can write
$$ \sum_{L \in \mathcal{D}} \gamma_L \alpha_L = \sum_{L \in \mathcal{D}} \int^{\infty}_{0} \chi(L,t)\,dt\;\alpha_L
= \int_0^{\infty}\Big (\sum_{L\in\mathcal{D}} \chi (L,t)\,\alpha_L \Big )dt,$$
where $\chi(L,t)=1$ for $ t < \gamma_L$ and zero otherwise, and the last equality follows by the monotone
convergence theorem.
Define $E_t= \{ x \in \mathbb{R} \; :\; F(x)>t\}$. Since $F$ is
assumed to be  a $v$-measurable function, $E_t$ is a $v$-measurable set for
every $t$. Moreover, since $F \in  L^1(v)$ we have, by
Chebychev's inequality, that the $v$-measure of $E_t$ is finite for all
real $t$.
 If $\;\chi(L,t)=1$ then $L \subset E_t$.
  Moreover, there is a collection of  maximal disjoint
dyadic intervals $\mathcal{P}_t$ that are contained in $E_t$.
Then we can write
\begin{equation}
   \sum_{L \in \mathcal{D}} \chi(L,t) \alpha_L  \leq \sum_{L \subset E_t} \alpha_L
   = \sum_{L \in \mathcal{P}_t} \sum_{I \in \mathcal{D}(L)} \alpha_I \leq  B \sum_{L \in \mathcal{P}_t}v(L) \leq Bv(E_t),
\end{equation}
where, in the second inequality, we used the fact that $\{
\alpha _I \}_{I \in \mathcal{D}}$ is a $v$-Carleson sequence with intensity $B$.
Thus we can estimate
\begin{equation*}
   \sum_{L \in \mathcal{D}} \gamma_L \alpha_L \leq B \int^{\infty}_{0} v(E_t ) dt = B \int_{\mathbb{R}} F(x)\,v(x)\,dx.
\end{equation*}
The last equality follows from the layer
cake representation.\\

\noindent ($\Leftarrow$) Assume (\ref{eqn:WCL}) is true; in particular it  holds for $F(x)={\chi_J(x)}/{|J|}$.  Since
$\inf_{x\in I} F(x)= 0$ if $I\cup J =\emptyset$, and $\inf_{x\in I} F(x)= {1}/{|J|}$ otherwise,
\[ \frac{1}{|J|}\sum_{I\in\mathcal{D}(J)} \alpha_I \leq \sum_{I\in\mathcal{D}}  \alpha_I\inf_{x\in I} F(x) \leq \int_{\mathbb{R}} F(x)\,v(x)\,dx=  m_J v.\]
\end{proof}

\subsubsection{Little Lemma}

The following Lemma was
proved by Beznosova in \cite{Be1} using the Bellman function
$\displaystyle{B(u,v,l)= u -{1}/{v(1+l)}}$.


\begin{lemma}[Little Lemma \cite{Be1}]\label{litlem}
Let $v$ be a weight, such that $v^{-1}$ is a a weight as well, and
let $\{ \lambda_I \}_{I \in \mathcal{D}}$ be a Carleson sequence with
intensity $B$. Then,  $\{{ \lambda_I}/{m_Iv^{-1}} \}_{I \in \mathcal{D}}$ is  a $v$-Carleson sequence with intensity $4B$, that is
for all $J\in \mathcal{D}$,
\[({1}/{|J|}) \sum_{I \in \mathcal{D}(J)}
{\lambda_I}/{m_I{v^{-1}}}\leq 4B \; m_Jv.\]
\end{lemma}
For a proof of this result we refer \cite[Prop. 3.4]{Be},  or
\cite[Prop. 2.1]{Be1}. For an $\R^N$ version of this result see \cite[Prop 4.6]{Ch1}.



Lemma~\ref{litlem} together with   Lemma~\ref{weightedCarlesonLem} immediately yield the following:

\begin{lemma}[\cite{NV}]\label{folklem}
Let $v$ be a weight such that $v^{-1}$ is also a weight.
Let $\{\lambda_{J}\}_{J \in \mathcal{D}}$ be a Carleson sequence
with intensity $B$, and let $F$ be a non-negative measurable function on the
line. Then,
\begin{equation*}
\sum_{J \in \mathcal{D}}({ \lambda_{J} }/{m_J v^{-1}}) \, \inf_{x \in J} F(x)
\leq C \;B \int_{\mathbb{R}}F(x)\,v(x)\,dx.
\end{equation*}
\end{lemma}

Note that Lemma~\ref{litlem}  can be deduced from Lemma~\ref{folklem}  with $F(x)=\chi_J(x)$.


\subsection{$\alpha$-Lemma}

The following lemma  was proved by Beznosova 
for $\alpha={1}/{4}$ in \cite{Be}, and  by Nazarov and Volberg for $0<\alpha <1/2$   in \cite{NV}, using the Bellman function
$\displaystyle{B(u,v)= (uv)^{\alpha}}$.



\begin{lemma}[$\alpha$-Lemma]\label{alphacoro} Let $w\in A_2^d$ and then  for any $\alpha \in (0,{1}/{2})$,
the sequence $\{ \mu^{\alpha}_I \}_{I \in \mathcal{D}}$, where
\begin{equation*}
  \mu^{\alpha}_I:= (\miw)^{\alpha}(m_I{w^{-1}})^{\alpha}|I| \bigg( \diws + \diwis \bigg),
  \end{equation*}
is a Carleson sequence with  intensity  $C_{\alpha}[w]_{A_2}^{\alpha}$,
with $C_{\alpha}={72}/({\alpha-2\alpha^2})$.
\end{lemma}


A proof of this lemma that works in $\R^N$ (for $\alpha=1/4$) can be found in \cite[Prop. 4.8]{Ch1}, and one
that works in  geometric doubling metric spaces can be found in \cite{NV1, V}.

The following lemmas simplify the exposition of the main theorems (this was pointed to us by one of our referees).
 We deduce these lemmas  from the $\alpha$-Lemma. According to our kind
anonymous referee,  one can also deduce Lemma~\ref{coro:referee}
 from  a pure Bellman-function argument without reference to the $\alpha$-Lemma.

\begin{lemma}\label{coro:referee}
Let $w\in A_2^d$ and let $\nu_I = |I| (m_Iw)^2(\Delta_I w^{-1})^2$. The sequence $\{\nu_I\}_{I\in\mathcal{D}}$ is
a Carleson sequence with intensity at most $C[w]_{A_2^d}^2$ for some numerical constant $C$ \emph{(}$C=288$ works\emph{)}.
\end{lemma}
\begin{proof} Multiply and divide $\nu_I$ by $(m_I w^{-1})^2$ to get for any  $0<\alpha <1/2$,
\[ \nu_I = |I| (m_Iw)^2  (m_I w^{-1})^2   {\big (|\Delta_I w^{-1}|/m_I w^{-1}\big )^2} 
\leq [w]_{A_2}^{2-\alpha}\mu^{\alpha}_I.\]
But $\{\mu^{\alpha}_I\}_{\mathcal{D}}$ is a Carleson sequence with intensity $C_{\alpha} [w]_{A_2}^{\alpha} $ by Lemma~\ref{alphacoro}, therefore
by Proposition~\ref{algcarseq}(i)
$\{\nu_I\}_{\mathcal{D}}$ is a Carleson sequence with intensity at most $C_{\alpha}[w]_{A_2^d}^2$ as claimed.
\end{proof}

It is well known that if $w\in A_2^d$ then $\{ |I||\Delta_Iw|^2/ (m_Iw)^2\}_{I\in \mathcal{D}}$
is a Carleson sequence with intensity $\log [w]_{A_2^d}$, see \cite{W}.
This estimate  together with Proposition~\ref{algcarseq}(i), give intensities
$[w]_{A_2^d}^{\alpha}\log [w]_{A_2^d}$ and $[w]_{A_2^d}^{2}\log [w]_{A_2^d}$
respectively for the sequences $\{\mu_I^{\alpha}\}_{I\in\mathcal{D}}$ and $\{\nu_I\}_{I\in\mathcal{D}}$.
The lemmas show we can improve the intensities by dropping the logarithmic factor.
Even more generally, we can show the following lemma, which extends the $\alpha$-Lemma~\ref{alphacoro} to the range
$\alpha\geq 1/2$. It also refines it for the range $\alpha \in (1/4,1/2)$ and shows that the blow up of the constant $C_{\alpha}$ for $\alpha=1/2$
is an artifact of the proof.
\begin{lemma}\label{lem:A2square}
Let $w\in A_2^d$,  $s>0$, and
 $$\tau^{s}_I := |I| (m_Iw)^{s}(m_Iw^{-1})^{s}\bigg( \diws + \diwis \bigg).$$
Then for $0< \alpha < \min\{1/2, s\}$,
 the sequence $\{\tau^s_I\}_{I\in\mathcal{D}}$ is
a Carleson sequence with intensity at most $C_{\alpha}[w]_{A_2^d}^s$ where  $C_{\alpha}$ is the constant in Lemma~\ref{alphacoro}
\emph{(}when $s>1/4$ can choose $\alpha =1/4$ and $C_{\alpha}=576.$\emph{)}
\end{lemma}



\subsection{Lift Lemma}

Given a dyadic interval $L$, and weights $u,v$, we introduce a family
of stopping
time intervals $\mathcal {ST}^m_L$ such that the averages  of the weights over any
stopping time interval $K \in \mathcal{ST}^m_L$ are comparable to the
averages on $L$, and $|K|\geq 2^{-m}|L|$. This construction appeared in \cite{NV} for the case $u=w$, $v=w^{-1}$.
We also present a lemma that lifts $w$-Carleson sequences on intervals to
$w$-Carleson sequences on ``$m$-stopping intervals".
We  present the proofs  for the convenience of the reader.

\begin{lemma}[Lift Lemma \cite{NV}] \label{liftlem}
Let $u$ and $v$  be  weights, 
$L$ be a dyadic interval and $m,n $ be  fixed natural numbers.  Let
$\mathcal{ST}^m_L$ be the
collection of maximal stopping time intervals $K \in
\mathcal{D}(L)$, where the stopping criteria are
 either {\em (i)} $\; |\Delta_Ku|/m_Ku + |\Delta_Kv|/m_Kv \geq {1}/{(m+n+2)}$,
  or  {\em (ii)} $\,|K| = 2^{-m}|L|$. Then for any stopping interval $K\in \mathcal{ST}^m_L$,
$\, e^{-1}m_Lu  \leq m_Ku \leq e\,m_Lu $, also $\, e^{-1}m_Lv \leq m_Kv \leq e\,m_Lv $.
\end{lemma}

Note that the roles of $m$ and $n$ can be interchanged and we get the family $\mathcal{ST}^n_L$
using the same stopping condition (i) as above, but with (ii) replaced by $|K|=2^{-n}|L|$.
Notice that $\mathcal{ST}^m_L$ is a partition of $L$ in dyadic subintervals of length at least $2^{-m}|L|$.
Any collection of subintervals of $L$ with this property will be an \emph{$m$-stopping time} for $L$.
\begin{proof}

Let $K$ be a  maximal stopping time interval; thus no dyadic interval
strictly bigger than $K$ can satisfy either stopping criteria.
If  $F$ is a dyadic interval strictly bigger than $K$ and contained
in $L$, then necessarily  
$ {|\Delta_F u|}/{m_F u}\leq (m+n+2)^{-1}$ and
${|\Delta_F v|}/{m_F v} \leq (m+n+2)^{-1}$.
This is particularly true for the parent of $K$. Let us denote by
$\widehat{K}$ the parent of $K$, 
 then
$ |m_Ku - m_{\widehat{K}}u| 
    = {|\Delta_{\widehat{K}}u|}/{2} \leq {m_{\widehat{K}} u}/{2(m+n+2)}.$
So,
$   m_{\widehat{K}} u \big(1 - {1}/{2(m+n+2)} \big) \leq m_Ku \leq m_{\widehat{K}} u \big(1 + {1}/{2(m+n+2)} \big).$
Iterating this process until we reach $L$, we will get that
\begin{equation*}
    m_Lu \bigg(1 - \frac{1}{2(m+n+2)} \bigg)^m \leq m_Ku \leq m_Lu \bigg(1 + \frac{1}{2(m+n+2)} \bigg)^m.
\end{equation*}
Remember that $|K| = 2^{-j}|L|$ where $ 0 \leq j \leq m$ so we will
iterate at most $m$ times. We can obtain the same bounds for $v$.  These clearly imply the estimates in the lemma,
since $\lim_{k\to\infty}(1+1/k)^k=e$.
\end{proof}

The following lemma lifts a $w$-Carleson sequence to $m$-stopping time intervals with comparable intensity.
The lemma appeared in \cite{NV} for the particular stopping time $\mathcal{ST}^m_L$  given by the stopping criteria (i) and (ii) in Lemma~\ref{liftlem},
and $w=1$.  This is a property of
any stopping time that stops once the $m^{th}$-generation is reached.
\begin{lemma}\label{corliftlemstop}
For each $L \in \mathcal{D}$, let $\mathcal{ST}^m_L$ be a partition of $L$ in dyadic subintervals
of length at least $2^{-m}|L|$ (in particular it could  be the stopping
time intervals defined in Lemma \ref{liftlem}). Assume $\{\nu_I \}_{I \in \mathcal{D}} $
is a $w$-Carleson sequence with intensity at most $A$, let $\nu^m_L :=
\sum _{K \in \mathcal{ST}^m_L} \nu_K$. Then $\{\nu_L^m\}_{L \in \mathcal{D}}$ is a $w$-Carleson sequence with
intensity at most $(m+1)A$.
\end{lemma}
\begin{proof}
In order to show that $\{\nu^m_L\}_{L\in \mathcal{D}}$ is a $w$-Carleson sequence with
intensity at most $(m+1)A$, it is enough to show that for any $J \in
\mathcal{D}$
$$\sum_{L \in \mathcal{D}(J)} \nu^m_L < (m+1)A \,w(J).$$
Observe that for each dyadic interval $K$ inside a fixed dyadic
interval $J$ there exist at most $m+1$ dyadic intervals $L$ such
that $K \in \mathcal{ST}^m_L$. Let us denote by $K^i$ the dyadic
interval that contains $K$ and such that $|K^i|=2^i|K|$.
If $K \in
\mathcal{D}(J)$  then  $L$ must be $K^0, K^1, ...   \; \mbox{or} \;  K^m$. We just
have to notice that if $L=K^i$, for $i>m$ then $K$ cannot be in
$\mathcal{ST}^{m}_L$ because $|K| < 2^{-m}|L|$. Therefore,
\begin{align*}
\sum_{L \in \mathcal{D}(J)} \nu^m_L  =&
\sum_{L \in \mathcal{D}(J)} \sum_{K \in
\mathcal{ST}^m_L} \nu_K  =  \sum_{K \in \mathcal{D}(J)}
\sum_{L \in \mathcal{D}(J) s.t. K \in \mathcal{ST}^m_L} \nu_K\\
 &\leq \sum_{K \in
\mathcal{D}(J)} (m+1)\nu_k \leq (m+1)A \, w(J).
\end{align*}
The last inequality follows by the definition of $w$-Carleson sequence
with intensity $A$.  The lemma is proved.
\end{proof}

\section{Paraproduct}
For $b\in BMO^d$, and   $m,n\in\mathbb{N}$,  a {\em  dyadic paraproduct of complexity  $(m,n)$}  is the operator defined
by
\begin{equation}\label{paraproduct(m,n)}
\big( \pi^{m,n}_{b} f \big)(x) := \sum_{L \in \mathcal{D}}
\sum _{(I,J)\in \mathcal{D}^n_m(L)} c^L_{I,J} m_I f \langle
b,h_I\rangle h_J(x),
\end{equation}
where $|c^L_{I,J}| \leq {\sqrt{|I|\,|J|}}/{|L|}$ for all
dyadic intervals $L$ and $(I,J)\in \mathcal{D}^n_m(L)$, where $\mathcal{D}^n_m(L)=  \mathcal{D}_n(L)\times  \mathcal{D}_m(L)$.

A dyadic paraproduct of complexity $(0,0)$ is the usual dyadic paraproduct $\pi_b$ known to be bounded in $L^p(\mathbb{R})$
if and only if $b\in BMO^d$.

A \emph{Haar shift operator of complexity $(m,n)$}, $m,n\in\mathbb{N}$,  is defined by
\[\big( S^{m,n} f \big)(x) := \sum_{L \in \mathcal{D}} \sum_{(I,J)\in \mathcal{D}^n_m(L)} c^L_{I,J} \langle
f,h_I\rangle h_J(x), \]
where $ |c^L_{I,J}| \leq {\sqrt{|I|\,|J|}}/{|L|}$.
Notice that the Haar shift operators are automatically uniformly bounded on $L^2(\mathbb{R} )$,
with operator norm less than or equal to one \cite{LPetR,CrMPz}.

 The dyadic paraproduct of complexity $(m,n)$ is the composition of $S^{m,n}$
 and $\pi_b$. Therefore, if $b\in BMO^d$ then $\pi_b^{m,n}$ is bounded in $L^2(\mathbb{R} )$,
since  $\pi_b^{m,n}=S^{m,n}\pi_b$, and both $\pi_b$ (the dyadic paraproduct)  and    $S^{m,n}$ (the
  Haar shift operators) are bounded in $L^2(\mathbb{R})$. 

Furthermore, $\pi_b$ and $S^{m,n}$ are bounded in $L^2(w)$ whenever $w\in A^d_2$.
Both of them 
obey bounds on $L^2(w)$ that are linear in the
$A_2$-characteristic of the weight,  immediately providing a quadratic bound
in the $A_2$-characteristic of the weight
for $\pi_b^{m,n}$. 
 We will show that in fact, the dyadic paraproduct of complexity $(m,n)$  obeys the same \emph{linear bound} in $L^2(w)$
 with respect to $[w]_{A_2^d}$ obtained by Beznosova \cite{Be1} for the dyadic paraproduct of complexity $(0,0)$, multiplied by
a polynomial factor that depends on the complexity.

The proof given by Nazarov and Volberg, in \cite{NV}, of the fact that Haar shift operators with complexity $(m,n)$ are
bounded in $L^2(w)$ with a bound that depends  linearly on the $A^d_2$-characteristic of $w$, and polynomially on the complexity,
 works, with appropriate modifications,
 for the dyadic paraproducts of complexity $(m,n)$.
Below we describe those modifications.  Beforehand, however, we will present this new and conceptually
simpler (in our opinion) proof for the linear bound on the $A^d_2$-characteristic for the dyadic paraproduct, which will
allow us to highlight  certain elements of the general proof without
 dealing  with  the complexity.


\subsection{Complexity $(0,0)$}\label{sec:paraproduct}

The dyadic paraproduct of complexity $(0,0)$  is defined by
$(\pi_b f)(x):= \sum _{I \in \mathcal{D}}c_I \; m_I f \; \langle
b,h_I\rangle h_I(x)$, {where}    $  |c_I| \leq 1$.

It is known  that  $\pi_b$  obeys a linear bound in $L^2(w)$
both in terms of the $A^d_2$-characteristic of the weight $w$ and the $BMO$-norm of $b$.

\begin{theorem}[\cite{Be1}]\label{thmA2paraproduct}
There exists $C>0$, such that for all $ b \in BMO^d$ and for all $w\in A^d_2$,
$$\|\pi_bf\|_{L^2(w)} \leq C [w]_{A^d_2} \|b\|_{BMO^d}\|f\|_{L^2(w)}.$$
\end{theorem}

Beznosova's  proof is based on the
$\alpha$-Lemma,  the Little Lemma (these were the new Bellman function ingredients that  she introduced),
and  Nazarov-Treil-Volberg's two-weight Carleson
embedding theorem, which can be found in \cite{NTV}.  Below, we give
another proof of this result; this proof is still based on the $\alpha$-Lemma~\ref{alphacoro} (via Lemma~\ref{coro:referee})
however it does not make use of the two-weight
Carleson embedding theorem. Instead we will use properties of Carleson
sequences such as the Little Lemma~\ref{litlem},  and the Weighted Carleson Lemma~\ref{weightedCarlesonLem}, following the argument in \cite{NV} for
Haar shift operators of complexity $(m,n)$. The extension of Theorem~\ref{thmA2paraproduct} to $\R^N$ can be found in
\cite{Ch1}, and the methods used there can be adapted to extend our proof  to $\mathbb{R}^N$ even in the
complexity $(m,n)$ case, see \cite{Mo1}.
\begin{remark}\label{remark:constants}Throughout the proofs a constant $C$ will be a numerical
constant that may change from line to line.
\end{remark}

\begin{proof}[Proof of Theorem~\ref{thmA2paraproduct}]
Fix $f \in L^2(w)$ and $g \in L^2(w^{-1})$. Define $b_I = \langle b , h_I \rangle$,
then  $\{b^2_I\}_{I\in\mathcal{D}}$ is a Carleson sequence with intensity $\|b\|_{BMO^d}^2$.

By duality, it suffices to prove:
\begin{equation}\label{dualityparaproduct}
 |\langle \pi_b(fw), gw^{-1}\rangle | \leq C \|b\|_{BMO^d} [w]_{A_2} \|f\|_{L^2(w)}\|g\|_{L^2(w^{-1})}.
\end{equation}
Note that
$\langle \pi_b(fw), gw^{-1}\rangle |=
 \big \langle \sum_{I \in \mathcal{D}}c_I b_I \mifw h_I,gw^{-1} \big \rangle  . $
Write $h_I = \alpha_I h^{\wi}_I + \beta_I {\chi_I}/{\sqrt{|I|}}$ where $\alpha_I = \alpha^{\wi}_I$ and
$\beta_I = \beta^{\wi}_I$ as described in Proposition
\ref{whaarbasis}. Then
\begin{equation}\label{dualitysumpara}
|\langle \pi_b(fw), gw^{-1}\rangle | \leq  \sum_{I \in \mathcal{D}}|b_I| \miafw \big |\big \langle gw^{-1}, \aI h_I^{w^{-1}}  + \bI \frac{\chi_I}{\sqrt{|I|}} \big \rangle \big |.
\end{equation}
Use the triangle inequality to break  the sum in (\ref{dualitysumpara}) into two sums to be estimated separately,
$| \langle \pi_b(fw), gw^{-1}\rangle | \leq \;\Sigma_1 +  \Sigma_2 .$
Where, using the estimates $|\alpha_I  |\leq \sqrt{\miwi} $, and $|\beta_I |\leq |\Delta_Iw^{-1}|/m_Iw^{-1}$,
\begin{eqnarray*}
\Sigma_1 &:=&  \sum_{I \in \mathcal{D}} |b_I| \miafw |\langle gw^{-1},h_I^{w^{-1}}  \rangle| \sqrt{\miwi} \\
\Sigma_2 & := & \sum_{I \in \mathcal{D}} |b_I| \miafw |\langle gw^{-1}, \chi_I \rangle| \diwi \frac{1}{\sqrt{|I|}}.
\end{eqnarray*}

\noindent {\bf Estimating $\Sigma_1$:}
First using that ${{m_I(|f|w)}/{\miw} \leq \inf_{x\in I} M_wf (x)}$,
and that $\langle gv,f\rangle = \langle g,f\rangle_v$;  second using the Cauchy-Schwarz inequality and $\miw \, \miwi \leq [w]_{A^d_2}$, we get\\
\begin{eqnarray*}
\Sigma_1
&\leq & \sum_{I \in \mathcal{D}} |b_I| \frac{ \inf_{x \in I} M_wf (x) }{\sqrt{\miwi}} \big |\langle g, h_I^{w^{-1}} \rangle_{w^{-1}} \big | \; \miwi  \; \miw \\
& \leq & [w]_{A^d_2} \bigg (\sum_{I \in \mathcal{D}} |b_I|^2
\frac{\inf_{x \in I} M^2_wf (x)}{\miwi} \bigg)^{\frac{1}{2}}\bigg
(\sum_{I \in \mathcal{D}}\big | \langle g, h_I^{w^{-1}} \rangle_{w^{-1}}\big |^2
\bigg)^{\frac{1}{2}}.
\end{eqnarray*}
Using  Weighted Carleson Lemma~\ref{weightedCarlesonLem}, with $F(x)=M^2_w f(x)$, $v=w$, and $\alpha_I={|b_I|^2}/{m_Iw^{-1}}$ (which
is a $w$-Carleson sequence with intensity $4\|b\|_{BMO^d}^2$, according to Lemma~\ref{litlem} ),
 together with the fact that $\{h_I^{w^{-1}}\}_{I\in \mathcal{D}}$ is an orthonormal system in $L^2(w^{-1})$, we get
\begin{eqnarray*}
\Sigma_1&\leq & 4[w]_{A^d_2} \|b\|_{BMO^d} \bigg( \int_{\mathbb{R}} M^2_{w}f(x) w(x) dx \bigg)^{\frac{1}{2}} \|g\|_{L^2(w^{-1})} \\
&\leq  &C [w]_{A^d_2} \|b\|_{BMO^d} \|f\|_{L^2(w)} \|g\|_{L^2(w^{-1})}.
\end{eqnarray*}
In the last inequality we used the fact that $M_w$ is bounded in
$L^2(w)$ with operator norm independent of $w$.

\vskip .1in
\noindent {\bf Estimating $\Sigma_2$:}
 Using  arguments similar to the ones
used for $\Sigma_1$, we conclude that,
\[\Sigma_2  = \sum_{I \in \mathcal{D}} |b_I| \wmiaf \; \wimiag \; \sqrt{\nu_I} 
 \leq   \sum_{I \in \mathcal{D}} |b_I| \sqrt{\nu_I} \; \inf_{x \in I} M_wf (x) M_{\wi}g(x),\]
where $\nu_I= |I| (m_Iw)^2(\Delta_Iw^{-1})^2$ as defined in Lemma~\ref{coro:referee}, and
in the last inequality we used that for any $I \in
\mathcal{D}$ and  all $x \in I$,
$$\wmiaf \,\wimiag \leq M_wf(x)M_{\wi}g(x). $$

Since $\{|b_I|^2\}_{I \in \mathcal{D}}$ and $\{\nu_I\}_{I \in \mathcal{D}}$ are Carleson sequences with intensities
$\|b\|^2_{BMO^d}$ and $C[w]^2_{A^d_2}$, respectively, by Proposition~\ref{algcarseq},
 the sequence $\{|b_I| \sqrt{\nu_I}\}_{I \in \mathcal{D}}$ is a Carleson
sequence with intensity $C\|b\|_{BMO^d}[w]_{A^d_2}$.  Thus, by Lemma~\ref{weightedCarlesonLem}
with $F(x)= M_wf(x)M_{\wi}g(x)$, $\alpha_I= |b_I|\sqrt{\nu_I}$, and $v=1$,
\[
\Sigma_2 \; \leq \; C \|b\|_{BMO^d}[w]_{A^d_2} \int_{\mathbb{R}} M_{w}f(x) M_{\wi}g(x) \,dx.
\]
Using the Cauchy-Schwarz inequality and  $w^{\frac{1}{2}}(x)w^{\frac{-1}{2}}(x)=1$ we get
\begin{eqnarray*}
\Sigma_2 &\leq & C [w]_{A^d_2} \|b\|_{BMO^d} \Big( \int_{\mathbb{R}} M^2_{w}f (x)w(x)dx \Big)^{\frac{1}{2}}\Big( \int_{\mathbb{R}} M^2_{\wi}g (x)\wi(x)dx \Big)^{\frac{1}{2}}\\
&= & C [w]_{A^d_2} \|b\|_{BMO^d} \|M_w f \|_{L^2(w)}\|M_{\wi}g \|_{L^2(\wi)} \\
&\leq& C [w]_{A^d_2} \|b\|_{BMO^d} \|f \|_{L^2(w)}\|g \|_{L^2(\wi)}.
\end{eqnarray*}
These estimates together give (\ref{dualityparaproduct}), and the  theorem is proved.
\end{proof}


\subsection{Complexity $(m,n)$}
In this section, we prove  an estimate for the dyadic paraproduct of
complexity $(m,n)$ that is linear in the $A_2$-characteristic and polynomial in the complexity.
The proof will follow the general lines of the argument presented in Section~\ref{sec:paraproduct}
for the complexity $(0,0)$ case,
with the added refinements devised by Nazarov and Volberg  \cite{NV},  adapted to our setting,  to handle the
general complexity.

\begin{theorem}\label{theoparcompmn}
Let $ b \in BMO^d$ and  $w\in A^d_2$. Then there is $C>0$  such that
$$\|\pi^{m,n}_bf\|_{L^2(w)} \leq C (n+m+2)^5[w]_{A^d_2} \|b\|_{BMO^d}\|f\|_{L^2(w)}.$$
\end{theorem}
\begin{proof}
Fix $f \in L^2(w)$ and $g \in L^2(w^{-1})$, define $b_I = \langle
b , h_I \rangle$ and let $C^n_m := (m+n+2)$. By duality, it is
enough to show that
\begin{equation*}
 |  \langle  \pi_b^{m,n}(fw), gw^{-1}\rangle  | \leq C(C^n_m)^5
[w]_{A^d_2}\|b\|_{BMO^d} \|g\|_{L^2(\wi)} \|f\|_{L^2(w)}.
\end{equation*}
We write the left-hand side as a  double sum, that we will estimate as
$$|  \langle  \pi_b^{m,n}(fw), gw^{-1}\rangle  |  \leq  \sum_{L \in \mathcal{D}}
\sum_{(I,J)\in \mathcal{D}^n_m(L)} |b_I|
\frac{\sqrt{|I|\,|J|}}{|L|} \miafw  |\langle gw^{-1}, h_J
\rangle|. $$
As before, we  write $h_J = \alpha_J h^{\wi}_J + \beta_J
{\chi_J}/{\sqrt{|J|}}$, with $\alpha_J = \alpha_J^{w^{-1}}$, $\beta_J=\beta_J^{w^{-1}}$, and break
the double sum  into two terms to be estimated
separately. Then
$\,|  \langle  \pi_b^{m,n}(fw), gw^{-1}\rangle  | \leq \Sigma_1^{m,n} + \Sigma_2^{m,n},$ where
\begin{align*}
\Sigma_1^{m,n}& := \sum_{L \in \mathcal{D}} \sum_{(I,J)\in \mathcal{D}^n_m(L)}
 |b_I| \frac{\sqrt{|I|\,|J|}}{|L|} \miafw |\langle g,h_J^{w^{-1}}\rangle_{w^{-1}}  | \sqrt{\mjwi},\\
\Sigma_2^{m,n} & := \sum_{L \in \mathcal{D}} \sum_{(I,J)\in \mathcal{D}^n_m(L)}
 |b_I| \frac{\sqrt{|I|}}{|L|}  \miafw |\langle gw^{-1}, \chi_J \rangle| \djwi .
\end{align*}
For a weight $v$, and  a locally integrable function $\phi $ we define the following quantities,
\begin{align}
S^{v,m}_L \phi &:= \sum_{J \in \mathcal{D}_m(L)} |\langle \phi,h_J^v
\rangle_v| \sqrt{m_Jv}\sqrt{|J|/|L|}, \label{def:SvmLphi}\\
R^{v,m}_L \phi &:= \sum_{J \in \mathcal{D}_m(L)} \frac{|\Delta_J
v|}{m_Jv} m_J(|\phi|v)\;{|J|}/{\sqrt{|L|}},\label{def:RvmLphi}\\
Pb^{v,n}_L \phi &:= \sum_{I \in \mathcal{D}_n(L)} |b_I| \; m_I(|\phi|v)
\sqrt{|I|/|L|}.\label{def:PbvnLphi}
\end{align}
For $s=1,2$ and $w\in A_2^d$, we also define the following Carleson sequences (see Lemma~\ref{corliftlemstop} and Lemma~\ref{lem:A2square}):

$\displaystyle{\mu^s_K := (\mkw)^{s}(\mkwi)^{s} \bigg(
\dkwis +\dkws \bigg)|K|}$, \\ with  intensity $C[w]^{s}_{A^d_2}$,

$\displaystyle{\mu^{m,s}_L := \sum_{K \in \mathcal{ST}^m_L} \mu_K} $,
 with intensity $C(m+1)[w]^{s}_{A^d_2}$,

$\displaystyle{\mu^{n,s}_L := \sum_{K \in \mathcal{ST}^n_L} \mu_K} $,
 with intensity $C(n+1)[w]^{s}_{A^d_2}$,

$\displaystyle{\mu^{b,s}_K := {|b_K|^2}(\mkw
\;\mkwi)^{s}}$, with   intensity
$\|b\|^2_{BMO^d}[w]^{s}_{A^d_2}$, and

$\displaystyle{\mu^{b,n,s}_L :=
\sum_{K \in \mathcal{ST}^n_L} \mu^{b,s}_K} $, with  intensity
$(n+1)\|b\|^2_{BMO^d}[w]^{s}_{A^d_2}$.\\
 Note that
 \[\Sigma_1^{m,n} \leq
\sum_{L \in \mathcal{D}} Pb^{w,n}_L f \; S^{\wi,m}_L g \quad \text{and} \quad \Sigma_2^{m,n}
\leq \sum_{L \in \mathcal{D}} Pb^{w,n}_L f \;  R^{\wi,m}_L g. \]
In
order to estimate $\Sigma_1^{m,n}$ and $\Sigma_2^{m,n}$ we will use the
following estimates for , $S^{\wi,m}_L g$, $R^{\wi,m}_L g$ and $Pb^{w,n}_L f$,
\begin{equation}
S^{\wi,m}_L g \leq \Big( \sum_{J \in \mathcal{D}_m(L)} |\langle
g,h_J^{\wi} \rangle_{\wi}|^2 \Big)^{\frac{1}{2}} (\mlwi)^{\frac{1}{2}},
\label{Sestpar}
\end{equation}
\begin{equation}
R^{\wi,m}_L g \leq C\, C^n_m
(\mlw)^{\frac{-s}{2}}(\mlwi)^{1-\frac{s}{2}} \inf_{x \in
L} \big (M_{\wi}(|g|^p)(x)\big )^{\frac{1}{p}} \sqrt{\mu^{m,s}_L}, \label{Restpar}
\end{equation}
\begin{equation}
Pb^{w,n}_L f  \leq C\, C^n_m
(\mlw)^{1-\frac{s}{2}}(\mlwi)^{\frac{-s}{2}} \inf_{x \in
L} \big (M_{w}(|f|^p)(x)\big )^{\frac{1}{p}}\nu_L^{n,s},\label{Pbestpar}
\end{equation}
where $\nu_L^{n,s}=\|b\|_{BMO^d}\sqrt{\mu^{n,s}_L} +
\sqrt{\mu^{b,n,s}_L}$, and
  $\displaystyle{p = 2 - (C_m^n)^{-1}}$ (note that $1<p<2$). In the proof it will become clear why this is a
good choice; the reader is invited to assume first that $p=2$ and reach a point  of no return in the argument.

Estimate (\ref{Sestpar}) is easy to show. We just use the
Cauchy-Schwarz inequality and the fact that $\mathcal{D}_m(L)$ is a partition of $L$.
$$
S^{\wi,m}_L g 
\leq   \Big( \sum_{J \in \mathcal{D}_m(L)} |\langle g ,h_J^{\wi}
\rangle_{\wi}|^2 \Big)^{\frac{1}{2}} \big(\mlwi \big)^{\frac{1}{2}}.
$$

Estimate (\ref{Restpar}) was obtained in \cite{NV}. With a
variation on their argument we prove estimate (\ref{Pbestpar}) in
Lemma \ref{Pblempar}.
Let us first use estimates \eqref{Sestpar}, \eqref{Restpar} and \eqref{Pbestpar} to estimate
 $\Sigma_1^{m,n}$ and $\Sigma_2^{m,n}$.\\

\noindent {\bf Estimate for $\Sigma_1^{m,n}$:}
 Use estimates \eqref{Sestpar}  and \eqref{Pbestpar} with $s=2$, the Cauchy-Schwarz inequality and the fact that
$\{h_J^{w^{-1}}\}_{J\in\mathcal{D}}$ is an orthonormal system in $L^2(w^{-1})$ and $\mathcal{D}=\cup_{L\in\mathcal{D}}\mathcal{D}_m(L)$. Then
$$
\Sigma_1^{m,n}
\leq C\,C^n_m  \Big(\sum_{L \in \mathcal{D}} \frac {(\nu^{n,2}_L)^2}{\mlwi}\inf_{x \in L}\big( M_{w}(|f|^p)(x)\big)^{\frac{2}{p}}\Big)^{\frac{1}{2}} \|g\|_{L^2(w^{-1})}.
$$
We will now use the Weighted Carleson Lemma~\ref{weightedCarlesonLem} with $F(x)= \big (M_w(|f|^p)(x)\big )^{2/p}$,   $v=w$, and
$\alpha_L={(\nu_L^{n,2} )^2 }/{m_Lw^{-1}}$. Recall that $\nu^{n,2}_L =\|b\|_{BMO^d}\sqrt{\mu^{n,2}_L} +
\sqrt{\mu^{b,n,2}_L}$. By Proposition \ref{algcarseq},  $\{(\nu_L^{n,2})^2\}_{L\in\mathcal{D}}$
is a Carleson sequence with intensity at most
$C\,C^n_m\|b\|^2_{BMO^d}[w]^{2}_{A^d_2}$. By Lemma~\ref{litlem}, $\{(\nu_L^{n,2} )^2 /m_Lw^{-1} \}_{L\in\mathcal{D}}$
is a $w$-Carleson sequence with comparable intensity. Thus  we will have that
\begin{align*}
\Sigma_1^{m,n} 
\leq& \;C\,(C^n_m)^{\frac{3}{2}} [w]_{A^d_2}\|b\|_{BMO^d} \|g\|_{L^2(\wi)} \Big \| M_{w}(|f|^p) \Big \|^{\frac{1}{p}}_{L^{\frac{2}{p}}(w)}\\
\leq& \; C \big[({2}/{p})'\big]^{\frac{1}{p}}(C^n_m)^{\frac{3}{2}} [w]_{A^d_2}\|b\|_{BMO^d} \|g\|_{L^2(\wi)} \big \| \, |f|^p \big \|^{\frac{1}{p}}_{L^{\frac{2}{p}}(w)}\\
=& \;C (C_m^n)^{\frac{5}{2}} [w]_{A^d_2}\|b\|_{BMO^d} \|g\|_{L^2(\wi)} \|f\|_{L^2(w)}.
\end{align*}
We used in the first inequality that $M_w$ is bounded in $L^q(w)$ for all $q>1$, more specifically we used that
$\|M_w f\|_{L^q(w)} \leq C q' \|f\|_{L^q(w)}$.
In our case ${q={2}/{p}}$ and
$q'= {2}/{(2-p)}=2C^n_m$.\\

\noindent {\bf Estimate for $\Sigma_2^{m,n}$:} Use estimates \eqref{Restpar}  and \eqref{Pbestpar} with $s=1$ in both cases,
together with the facts that $(m_Iw\, m_Iw^{-1} )^{-1}\leq 1$, and that the product of the infimum of positive quantities is smaller
than the infimum of the product. Then 
$$\Sigma_2^{m,n}  
\leq C(C^n_m)^2     \sum_{L \in \mathcal{D}} \nu^{n,1}_L\sqrt{\mu^{m,1}_L}\inf_{x \in L} \big (M_{w}(|f|^p)(x)\big )^{\frac{1}{p}}
 \big (M_{\wi}(|g|^p)(x)\big )^{\frac{1}{p}}.
$$
Since $(\nu^{n,1}_L)^2$ and $\mu^{m,1}_L$ have intensity at most $C(n+1)[w]_{A^d_2}\|b\|^2_{BMO}$ and $C(m+1)[w]_{A^d_2}$,
by Proposition \ref{algcarseq}, we have that
$\nu_L^{n,1} \sqrt{\mu^{m,1}_L}$ is a Carleson sequence with intensity at
most  $C\,C^m_n\|b\|_{BMO^d}[w]_{A^d_2}$. If we
now apply Lemma~\ref{weightedCarlesonLem} with $F^p(x)= M_w(|f|^p)(x) M_{w^{-1}}(|g|^p)(x)$,
$\alpha_L= \nu_L^{n,1}\sqrt{\mu_L^{m,1}}$, and $v=1$, we will have, by the Cauchy-Schwarz inequality and the boundedness of $M_v$
in $L^q(v)$ for $q=p/2>1$,
\begin{align*}
\Sigma_2^{m,n}
\leq& \; C(C^n_m)^3 [w]_{A^d_2} \|b\|_{BMO^d} \int_{\mathbb{R}} \big (M_{w}(|f|^p)(x)\big )^{\frac{1}{p}}\big (M_{\wi}(|g|^p)(x)\big )^{\frac{1}{p}} dx\\
\leq& \; C(C^n_m)^3 [w]_{A^d_2} \|b\|_{BMO^d}\big\| M_{w}(|f|^p) \big \|^{\frac{1}{p}}_{L^{\frac{2}{p}}(w)}\big \|M_{\wi}(|g|^p) \big\|^{\frac{1}{p}}_{L^{\frac{2}{p}}(\wi)}\\
\leq& \; C\big[({2}/{p})'\big]^{\frac{2}{p}}(C^n_m)^3 [w]_{A^d_2} \|b\|_{BMO^d}\big \| |f|^p \big\|^{\frac{1}{p}}_{L^{\frac{2}{p}}(w)}\big \| |g|^p \big \|^{\frac{1}{p}}_{L^{\frac{2}{p}}(\wi)}\\
=& \;C (C^n_m)^5 [w]_{A^d_2} \|b\|_{BMO^d}\|f\|_{L^2(w)} \|g\|_{L^2(\wi)}.
\end{align*}
Together these estimates prove the theorem, under the assumption that  estimate (\ref{Pbestpar}) holds.
\end{proof}


\subsection{Key Lemma}

The missing step in the previous proof is estimate (\ref{Pbestpar}), which we now prove.
The argument we present is an adaptation of the argument used in \cite{NV} to obtain estimate (\ref{Restpar}).

\begin{lemma}\label{Pblempar}
Let $b\in BMO^d$,  and let $\phi$ be a locally integrable function. Then,
$$Pb^{w,n}_L \phi \leq C\,C_m^n
(\mlw)^{1-\frac{s}{2}}(\mlwi)^{\frac{-s}{2}} \inf_{x \in
L}\big ( M_{w}(|\phi |^p)(x)\big )^{\frac{1}{p}} \nu_L^{n,s},$$
where $\nu_L^{n,s}=\|b\|_{BMO^d}\sqrt{\mu^{n,s}_L} + \sqrt{\mu^{b,n,s}_L}$,
and $p=2-(C^n_m)^{-1}$.

\end{lemma}

\begin{proof}
Let $\mathcal{ST}^n_L$ be the collection of stopping time intervals
defined in Lemma \ref{liftlem}. Noting that $\mathcal{D}_n(L)=\cup_{K\in\mathcal{ST}_L^n}\big (\mathcal{D}(K)\cap \mathcal{D}_n(L)\big )$, we get,
\begin{equation*}
Pb^{w,n}_L \phi
 = \sum_{K \in \mathcal{ST}^{n}_L}
\sum_{I \in \mathcal{D}(K) \bigcap \mathcal{D}_n(L)}
{|b_I|}\; m_I(|\phi |w) \sqrt{|I|/|L|}.
\end{equation*}
Note that if $K$ is a stopping time interval by the first
criterion then
\begin{eqnarray*}
Pb^{w,n}_L \phi & \leq &  \|b\|_{BMO^d}\, m_K(|\phi |w) {|K|}/{\sqrt{|L|}} \\
 & \leq & C_m^n \|b\|_{BMO^d}\,m_K (|\phi |w)
(\sqrt{|K|/|L|})\sqrt{2\mu^s_K} \,(\mkw \, \mkwi)^{\frac{-s}{2}}.
\end{eqnarray*}
The first inequality is true because ${|b_I|}/{\sqrt{|I|}}
\leq \|b\|_{BMO^d}$ and the second one because
$$1 \leq C_m^n \bigg (\dkw + \dkwi\bigg )\sqrt{|K|}\leq  C_m^n \sqrt{2\mu^{s}_K}(\mkw\;\mkwi)^{\frac{-s}{2}}.$$
Now we use the fact, proved in Lemma \ref{liftlem}, that we can
compare the averages of the weights on the stopping intervals with
their averages in $L$, paying a price of a constant $e$, and continue estimating by
\[\sqrt{2}C^m_n e^{s}\|b\|_{BMO^d}m_K (|\phi |w)
\sqrt{|K|/|L|}\sqrt{\mu^s_K} (\mlw \, \mlwi)^{\frac{-s}{2}}.\]

If $K$ is a stopping time interval by the second criterion, then the
sum collapses to just one term
\begin{align*}
\sum_{I \in \mathcal{D}(K) \bigcap \mathcal{D}_n(L)}|b_I| \;
& m_I(|\phi | w){\sqrt{|I|/|L|}} \; = \; {|b_K|}
\;
m_K(|\phi |w) {\sqrt{|K|/|L|}} \\
=\; &  \; m_K (|\phi |w)
\sqrt{|K|/|L|}\sqrt{\mu^{b,s}_K}\, (\mkw \, \mkwi)^{\frac{-s}{2}}\\
\leq\; & \; C^m_n e^{s}m_K (|\phi |w)
\sqrt{|K|/|L|}\sqrt{\mu^{b,s}_K}\, (\mlw \, \mlwi)^{\frac{-s}{2}}.
\end{align*}
Let  $\Xi_1 (L): = \{ K \in \mathcal{ST}^n_L : K \text{ is a
stopping time interval by criterion 1} \}$, and
$\Xi_2(L) := \{ K
\in \mathcal{ST}^n_L : K \text{ is a stopping time interval by
criterion 2} \}.$
 Note that $\Xi_1(L) \bigcup \Xi_2(L)$ is a
partition of $L$. We then have,
\begin{equation}\label{PbLwn}
Pb^{w,n}_L \leq \sqrt{2} C^m_n e^{s}\;
(\mlw \, \mlwi)^{\frac{-s}{2}}
\Big(\|b\|_{BMO^d} \Sigma^1_{Pb} + \Sigma^2_{Pb} \Big),
\end{equation}
where the terms $\Sigma^1_{Pb}$ and $\Sigma^2_{Pb}$ are defined as follows,
\begin{eqnarray*}
 \Sigma^1_{Pb} & := &\sum_{K \in \Xi_1(L)} m_K (|\phi |w){\sqrt{|K|/|L|}}\sqrt{\mu^s_K},  \\
\Sigma^2_{Pb} & := & 
\sum_{K \in \Xi_2(L)} m_K (|\phi |w)
{\sqrt{|K/|L|}}\sqrt{\mu^{b,s}_K}.
\end{eqnarray*}

Now estimate $\Sigma^1_{Pb}$ using the Cauchy-Schwarz inequality,  noting that  we can
move a power ${p}/{2} <1$ from outside to inside the sum, and that
$\mu^{n,s}_L: = \sum_{I \in \mathcal{ST}^n_L} \mu^s_K \geq \sum_{I \in \Xi_1(L)} \mu^s_K$,
\begin{eqnarray}
\Sigma^1_{Pb} &\leq& \Big (\sum_{K \in \Xi_1(L)} (m_K (|\phi |w))^2
{|K|}/{|L|}\Big)^{\frac{1}{2}}
\Big(\sum_{K \in \Xi_1(L)}  \mu^s_K\Big)^{\frac{1}{2}} \nonumber \\
&\leq& \Big(\sum_{K \in \Xi_1(L)} (m_K (|\phi |w))^p
\big({|K|}/{|L|}\big)^{\frac{p}{2}}\Big)^{\frac{1}{p}}
\sqrt{\mu^{n,s}_L} \label{sigmapb}.
\end{eqnarray}
By the second stopping criterion
$|K|/{|L|}= 2^{-j}$ for $0 \leq j \leq m$, then
\begin{equation}\label{whatever}
\big({|K|}/{|L|}\big)^{\frac{p}{2}}  = 2^{-j + \frac{j}{2(m+n+2)}}
< 2\cdot 2^{-j}= 2{|K|}/{|L|}.
\end{equation}

Plugging  \eqref{whatever} into \eqref{sigmapb} gives
\[
\Sigma^1_{Pb} \leq \Big(2\sum_{K \in \Xi_1(L)} (m_K (|\phi |w))^p
{|K|}/{|L|}\Big)^{\frac{1}{p}}\sqrt{\mu^{n,s}_L}.\]
Use  H\"older's inequality inside the sum, then Lift Lemma~\ref{liftlem}, to get
\begin{eqnarray*}
\Sigma^1_{Pb}
&\leq& \Big (\sum_{K \in \Xi_1(L)} (m_K (|\phi |^p w))
(\mkw)^{p-1} {|K|}/{|L|}\Big)^{\frac{1}{p}}\sqrt{\mu^{n,s}_L} \\
&\leq& 2^{\frac{1}{p}} (e\,\mlw)^{1-\frac{1}{p}}
\Big(\frac{1}{|L|}\sum_{K \in \Xi_1(L)} \int_K
|\phi (x)|^p w(x)\, dx \Big)^{\frac{1}{p}}\sqrt{\mu^{n,s}_L}.
\end{eqnarray*}
Observe that the intervals $K\in \Xi_1(L)$ are disjoint subintervals of $L$, therefore,
$\sum_{K \in \Xi_1(L)} \int_K |\phi (x)|^p w(x)\, dx \leq \int_L |\phi (x)|^p w(x)\, dx$, thus,
\begin{equation}
\Sigma^1_{Pb}
\leq  2e\, \mlw \inf_{x \in L} \big (M_w (|\phi |^p)(x)
\big)^{\frac{1}{p}}\sqrt{\mu^{n,s}_L}. \label{S1Pb}
\end{equation}
Similarly we estimate $\Sigma^2_{Pb}$, to get,
\begin{align*}
\Sigma^2_{Pb} &\leq \Big(\sum_{K \in \Xi_2} (m_K (|\phi |w))^2
{|K|}/{|L|}\Big)^{\frac{1}{2}}
\Big(\sum_{K \in \Xi_2}  \mu^{b,s}_K\Big)^{\frac{1}{2}}\\
&\leq \; \Big(\sum_{K \in \mathcal{ST}^n_L} (m_K (|\phi |w))^p
\big({|K|/|L|}\big)^{\frac{p}{2}}\Big)^{\frac{1}{p}}
\sqrt{\mu^{b,n,s}_L}.
\end{align*}
Following the same steps as we did in the estimate for
$\Sigma^1_{Pb}$, we will have
\begin{equation}\label{S2Pb}
 \Sigma^2_{Pb} \leq 2e \, \mlw \inf_{x \in L} \Big(M_w (|\phi |^p)(x)
\Big)^{\frac{1}{p}} \sqrt{\mu^{b,n,s}_L}.
\end{equation}

Insert estimates (\ref{S1Pb}) and (\ref{S2Pb}) into (\ref{PbLwn}).
Altogether, we can bound $ Pb^{w,n}_L$ by
$$ C\, C^m_n e^{s+1}\;
(\mlw)^{1-\frac{s}{2}} (\mlwi)^{\frac{-s}{2}}\inf_{x \in
L} \Big(M_w (|\phi |^p)(x)
\Big)^{\frac{1}{p}}\nu^{n,s}_L.
$$
The lemma is proved.
\end{proof}

\begin{remark}
In \cite{NV1}, Nazarov and Volberg extend the results that they had  for Haar shift operators
in \cite{NV} to  metric spaces with geometric
doubling.
One can  extend  
Theorem~\ref{theoparcompmn}   to this  setting as well, 
see \cite{Mo1}.
\end{remark}

\section{Haar Multipliers}

For a weight $w$, $t\in\mathbb{R}$, and $m,n\in \mathbb{N}$,  a  {\em $t$-Haar multiplier of complexity $(m,n)$} is the
operator defined as
\begin{equation}
 T^{m,n}_{t,w} f (x) := \sum_{L \in \mathcal{D}} \sum_{(I,J)\in \mathcal{D}^n_m(L)}
 c^L_{I,J}\, \Big(\frac{w(x)}{m_L w}\Big)^t \langle f,h_I\rangle h_J(x),
\end{equation}
where $|c^L_{I,J} |\leq {\sqrt{|I|\,|J|}}/{|L|}$.

Note that these operators have symbols,
namely  $c^L_{I,J} \big({w(x)}/{m_L w}\big)^t$,
that depend on: the space variable $x$, the  frequency encoded in the dyadic interval $L$, and the complexity encoded in the subintervals
$I\in\mathcal{D}_n(L)$ and $J\in\mathcal{D}_m(L)$.  This makes these operators
akin to pseudodifferential operators where the trigonometric functions have been replaced by the Haar functions.

Observe that  $T_{t,w}^{m,n}$ is different from both $S^{m,n}T^t_w$ and $T^t_wS^{m,n}$ and, that,
unlike $T^{m,n}_{t,w}$,  both $S^{m,n}T^t_w$ and $T^t_wS^{m,n}$
 obey the same bound that $T^t_w$ obeys in $L^2(\mathbb{R})$, because  the Haar shift multipliers
have $L^2$-norm less than or equal to one.

\subsection{Necessary conditions}

Let us first show a
necessary condition on the weight $w$ so that  the Haar multiplier $T^{m,n}_{w,t}$  with $c^L_{I,J}=\sqrt{|I|\,|J|}/|L|$ is bounded on
$L^p(\mathbb{R})$. This necessary $C^d_{tp}$-condition is the same condition
found in \cite{KP} for the $t$-Haar multiplier of complexity $(0,0)$.


\begin{theorem}\label{neccondhaarmult}
Let $w$ be a weight, $m,n$  positive integers and $t$  a real number.  If
$T^{m,n}_{t,w}$ is the $t$-Haar multiplier with $c^L_{I,J}=\sqrt{|I|\,|J|}/|L|$ and
is a bounded operator in $L^p(\mathbb{R})$, then $w$ is
in $C^d_{tp}.$
\end{theorem}
\begin{proof}
Assume that $T^{m,n}_{t,w}$  is bounded in $L^p(\mathbb{R})$ for $1<p<\infty$. Then
there exists $C>0$ such that for any $f \in L^p(\mathbb{R})$ we have
$\|T^{m,n}_{t,w}f\|_{p}\leq C\|f\|_{p}. $
Thus for any $I_0 \in \mathcal{D}$ we should have
\begin{equation}\label{proofsufcond1}
\|T^{m,n}_{t,w}h_{I_0}\|^p_{p}\leq C^p\|h_{I_0}\|^p_{p}.
\end{equation}
Let us compute  the norm on the left-hand side of (\ref{proofsufcond1}).
Observe that
\begin{equation}\label{proofsufcond3}
T^{m,n}_{t,w}h_{I_0}(x) = \sum_{L
\in \mathcal{D}} \sum_{(I,J)\in \mathcal{D}^n_m(L)}{\sqrt{|I|\,|J|}}/{|L|}
 \Big(\frac{w(x)}{m_L w}\Big)^t \langle
h_{I_{0}},h_I\rangle h_J(x).
\end{equation}

We have  $\langle h_{I_{0}},h_I\rangle
= 1$ if $I_{0}=I$ and $\langle h_{I_{0}},h_I\rangle = 0$ otherwise.
Also, there exists just one dyadic interval $L_{0}$ such that $I_{0}
\subset L_{0}$ and $|I_{0}| = 2^{-n}|L_{0}|$. Therefore we can
collapse the sums in \eqref{proofsufcond3} to just one sum, and calculate
the $L^p$-norm as follows,
\[\|T^{m,n}_{t,w}h_{I_0}\|^p_{p} = \int_{\mathbb{R}}\Big|
 \sum _{ J \in
\mathcal{D}_m(L_{0})}
{\sqrt{|I_{0}|\, |J|}}/{|L_{0}|}\Big(\frac{w(x)}{m_{L_{0}}
w}\Big)^t  h_J(x) \Big|^p dx. \]
Furthermore, since $\mathcal{D}_m(L_0)$ is a partition of $L_0$, the power $p$ can be put
inside the sum, and we get,
\begin{equation}\label{proofsufcond4}
  \|T^{m,n}_{t,w}h_{I_0}\|^p_{p}
= \big ({|I_{0}|^{\frac{p}{2}}}/{|L_{0}|^{p-1}}\big )
\big ({m_{L_{0}}w^{tp}}/{(m_{L_0}w)^{pt}}\big ).
\end{equation}

Inserting
$\|h_{I_0}\|^p_{p}=|I_{0}|^{1-\frac{p}{2}}$ and \eqref{proofsufcond4}  in
\eqref{proofsufcond1}, we will have that for any dyadic interval
$I_{0}$ there exists $C$ such that
\begin{equation*}
\big ({|I_{0}|^{\frac{p}{2}}}/{|L_{0}|^{p-1}}\big )
\big ({m_{L_{0}}w^{tp}}/{(m_{L_0}w)^{pt}}\big ) \leq C^p |I_{0}|^{1-\frac{p}{2}}.
\end{equation*}
Thus,
$ { m_{L_{0}}w^{tp}}/{(m_{L_0}w)^{pt}} \leq C^p
|I_{0}|^{1-p}|L_{0}|^{p-1} = 
C^p 2^{n(p-1)}=:C_{n,p}.$
Now observe that this inequality  should hold for any $L_{0} \in
\mathcal{D}$, we just have to choose as $I_{0}$ any of the
descendants of $L_{0}$ in the $n$-th generation, and that $n$
is fixed. Therefore,
$${\; [w]_{C^d_{2t}}=\sup_{L \in \mathcal{D}} (m_{L}w^{tp} ) (\mlw )^{-pt} \leq C_{n,p}.}$$
We conclude that $\; w \in C^d_{tp}$;
moreover
$\displaystyle{\;[w]_{C^d_{tp}} \leq  2^{n(p-1)} ||T^{m,n}_{t,w}||^p_{p}.}$
\end{proof}


\subsection{Sufficient condition}

For most $t\in\R$, the $C^d_{2t}$-condition is not only necessary but also sufficient  for a $t$-Haar multiplier
of complexity $(m,n)$ to be bounded on $L^2(\mathbb{R} )$; this was proved in \cite{KP} for the case $m=n=0$.
Here we are concerned not only with the boundedness but also with the dependence of the operator norm
on the $C^d_{2t}$-constant. For the case $m=n=0$ and $t=1, \pm 1/2$
this was studied in \cite{P2}. Beznosova \cite{Be} was able to obtain estimates, under the additional condition on the weight:
$w^{2t}\in A^d_{p}$ for some $p>1$,  for the case of complexity $(0,0)$ and for all $t\in\mathbb{R}$.  We generalize her
results when $w^{2t}\in A^d_2$ for complexity $(m,n)$. Our proof differs from hers in that we are adapting the methods
of Nazarov and Volberg \cite{NV} to this setting as well. Both proofs rely on the $\alpha$-Lemma (Lemma~\ref{alphacoro})
and on the Little Lemma (Lemma~\ref{litlem}). See also \cite{BeMoP}.

\begin{theorem}\label{sufcondhaarmult}
Let $t$ be a real number and $w$ a weight in $C^d_{2t}$, such that
$w^{2t} \in A^d_2$. Then $T^{m,n}_{t,w}$, a $t$-Haar multiplier
with depth $(m,n)$, is bounded in $L^2(\mathbb{R} )$. Moreover,
$$\|T^{m,n}_{t,w}f \|_{2} \leq C(m+n+2)^3 [w]^{\frac{1}{2}}_{C^d_{2t}}[w^{2t}]^{\frac{1}{2}}_{A^d_2}\|f\|_2.$$
\end{theorem}

\begin{proof}
Fix $f,\; g \in L^2(\mathbb{R})$.
By duality, it is enough to show that
\begin{equation*}
| \langle  T^{m,n}_{t,w}f,g \rangle | \leq C (m+n+2)^3
[w]^{\frac{1}{2}}_{C^d_{2t}}[w^{2t}]^{\frac{1}{2}}_{A^d_2}\|f\|_{2}
\|g\|_{2}.
\end{equation*}
The inner product on the left-hand-side can be expanded into a double sum that we now estimate,
$$| \langle  T^{m,n}_{t,w}f,g \rangle |
 \leq \sum_{L \in \mathcal{D}} \sum_{(I,J)\in \mathcal{D}^n_m(L)} ({\sqrt{|I|\,|J|}}/{|L|})\;
\frac{ |\langle f,h_I\rangle|}{(m_L w)^t}\;| \langle gw^t  , h_J \rangle|.
$$

Decompose $h_J$ into a linear combination of a
weighted Haar function and a characteristic function, $h_J =
\alpha_J h^{w^{2t}}_J + \beta_J {\chi_J}/{\sqrt{|J|}}$, where
$\alpha_J = \alpha^{w^{2t}}_J$, $\beta_J = \beta^{w^{2t}}_J$,
$|\alpha_J|\leq \sqrt{m_Jw^{2t}}$, and $|\beta_J|\leq {|\Delta_J(w^{2t})|}/{m_Jw^{2t}}$.
Now we break this sum  into two terms to be estimated separately so that,
$$| \langle  T^{m,n}_{t,w}f,g \rangle | \leq  \Sigma_3^{m,n} + \Sigma_4^{m,n},$$
where
\begin{align*}
\Sigma_3^{m,n} :=& \sum_{L \in \mathcal{D}} \sum_{(I,J)\in \mathcal{D}^n_m(L)} \frac{\sqrt{|I|\,|J|}}{|L|}
\frac{\sqrt{m_{J}(w^{2t})}}{(m_L w)^t}  |\langle f,h_I\rangle| \; | \langle gw^{t}  , h^{w^{2t}}_J \rangle |,\\
 \Sigma_4^{m,n}:=&  \sum_{L \in \mathcal{D}} \sum_{(I,J)\in \mathcal{D}^n_m(L)}
 \frac{|J| \sqrt{|I|} }{|L|(m_L w)^t} \frac{|\Delta_J (w^{2t})|}{m_{J}(w^{2t})} |\langle f,h_I\rangle| \; m_J( |g| w^t ) .
\end{align*}

Again, let $p = 2 - (C_n^m)^{-1}$, and define as in (\ref{def:SvmLphi}) and (\ref{def:RvmLphi}), the quantities $S_L^{v,m}\phi$
and $R_L^{v,m}\phi $, with $v=w^{2t}$.
Let
\[ P^n_L \phi := \sum_{I \in \mathcal{D}_n(L)} \; |\langle f, h_I
\rangle| {\sqrt{|I|/|L|}},\]
and
$$\eta_I := m_I (w^{2t}) \; m_I
(w^{-2t}) \bigg(
\frac{|\Delta_I(w^{2t})|^2}{|m_I w^{2t}|^2} + \frac{|\Delta_I(w^{-2t})|^2}{|m_I w^{-2t}|^2} \bigg)|I|.$$
 By Lemma~\ref{lem:A2square} with $s=1$,  $\{\eta_I\}_{I\in\mathcal{D}}$ is  a Carleson sequence
 with intensity $C[w^{2t}]_{A^d_2}$. Let
${\eta^m_L := \sum_{I \in \mathcal{ST}_L^m} \eta_I},$
 where the stopping time $\mathcal{ST}_L^m$ is defined as in Lemma~\ref{liftlem} (with respect to the weight $w^{2t}$).
By Lemma~\ref{corliftlemstop},  $\{\eta_L^m\}_{L\in\mathcal{D}}$ is a Carleson sequence with  intensity $C(m+1)[w^{2t}]_{A^d_2}$.


Observe that on the one hand $\langle gw^t,h^{w^{2t}}_J\rangle = \langle gw^{-t},h^{w^{2t}}_J \rangle_{w^{2t}} $, and  on the other $m_J(|g|w^t)=m_J(|gw^{-t}|w^{2t})$.
Therefore,
 $$\Sigma_3^{m,n}=\sum_{L \in \mathcal{D}}{(m_Lw)^{-t}}S^{w^{2t},m}_L
(gw^{-t}) \; P_L^nf ,$$

$$\Sigma_4^{m,n} = \sum_{L \in \mathcal{D}} {(m_Lw)^{-t}}
 R^{w^{2t},m}_L (gw^{-t}) \; P^n_L f.$$

\noindent 
Estimates (\ref{Sestpar})  and (\ref{Restpar}) with $s=1$ hold for  $S^{w^{2t},m}_L (gw^{-t})$ and $R^{w^{2t},m}_L (gw^{-t})$,  with $w^{-1}$ and $g$ replaced by $w^{2t}$ and $gw^{-t}$:
\begin{eqnarray*}
S^{w^{2t},m}_L (gw^{-t})
& \leq& (m_L w^{2t})^{\frac{1}{2}}\Big(\sum_{J \in \mathcal{D}_m(L)}
|\langle gw^{-t},h_J^{w^{2t}} \rangle_{w^{2t}} |^2\Big)^{\frac{1}{2}},\\
R^{w^{2t},m}_L (gw^{-t}) & \leq & C\,C^n_m (m_L
w^{2t})^{\frac{1}{2}}(m_L w^{-2t})^{\frac{-1}{2}}
F^{\frac12}(x)
\sqrt{\eta^m_L},
\end{eqnarray*}
where $F(x)=\inf_{x \in L} \big ( M_{w^{2t}}(|gw^{-t}|^p)(x)\big )^{\frac{2}{p}}$.
Estimating $P_L^n f$ is simple:
\[P_L^n f 
\leq \Big( \sum_{I \in \mathcal{D}_n(L)} {|I|}/{|L|} \Big)^{\frac{1}{2}} \Big( \sum_{I \in \mathcal{D}_n(L)} |\langle f,h_I \rangle|^2  \Big)^{\frac{1}{2}}
= \Big( \sum_{I \in \mathcal{D}_n(L)} |\langle f,h_I \rangle|^2
\Big)^{\frac{1}{2}}.\]


\noindent{\bf Estimating $\Sigma_3^{m,n}$:} Plug in the estimates for $S^{w^{2t},m}_L (gw^{-t})$ and $P_L^nf $,
observing that ${(m_Lw^{2t})^{\frac{1}{2}}}/{(m_L w)^t}\leq [w]_{C^d_{2t}}^{\frac12}$.
Using the Cauchy-Schwarz inequality, we get,
\begin{align*}
\Sigma_3^{m,n} 
& \; \leq  \sum_{L \in \mathcal{D}} [w]_{C^d_{2t}}^{\frac12} \Big(\sum_{J \in
\mathcal{D}_m(L)} |\langle
gw^{-t},h_J^{w^{2t}} \rangle_{w^{2t}}|^2 \Big)^{\frac{1}{2}} \Big( \sum_{I \in \mathcal{D}_n(L)} |\langle f,h_I \rangle|^2   \Big)^{\frac{1}{2}} \\
& \leq \; [w]_{C^d_{2t}}^{\frac{1}{2}}  \|f\|_{2} \| gw^{-t}\|_{L^2(w^{2t})}
\; = \;   [w]_{C^d_{2t}}^{\frac{1}{2}} \|f\|_{2}\|g\|_{2}.
\end{align*}

\noindent{\bf Estimating $\Sigma_4^{m,n}$:} Plug in the estimates for $R^{w^{2t},m}_L (gw^{-t})$ and $P_L^nf $,
where $F(x)=\big ( M_{w^{2t}}(|gw^{-t}|^p)(x)\big )^{2/p}$. Using the Cauchy-Schwarz inequality and considering again that
${(m_L w^{2t})^{\frac12}}/{(m_L w)^t}\leq [w]_{C^d_{2t}}^{\frac12}$, then
$$\Sigma_4^{m,n}
 \leq C\, C^n_m [w]_{C^d_{2t}}^{\frac{1}{2}}  \|f\|_2\Big(\sum_{L \in \mathcal{D}}\frac{\eta_L^m}{m_L w^{-2t}}
\inf_{x \in L} F(x)
 \Big)^{\frac{1}{2}}.
$$
Now, use the Weighted Carleson Lemma \ref{weightedCarlesonLem} with $\alpha_L={\eta_L^m}/{m_L(w^{-2t} )}$
(which by Lemma~\ref{litlem} is a $w^{2t}$-Carleson sequence with intensity at most
$C\,C^n_m[w^{2t}]_{A^d_2}$). Let $F(x)=\big ( M_{w^{2t}}|gw^{-t}|^p(x)\big )^{2/p}$, and $v=w^{2t}$, then
$$
\Sigma_4^{m,n}
 \leq C(C^n_m)^2 [w]_{C^d_{2t}}^{\frac{1}{2}}
[w^{2t}]^{\frac{1}{2}}_{A^d_2} \|f\|_{2} \big \|M_{w^{2t}}(|gw^{-t}|^p)\big \|^{\frac{1}{p}}_{L^{\frac{2}{p}}(w^{2t})}.
$$
Using \eqref{bddmaxfct}, that is the boundedness of $M_{w^{2t}}$ in $L^{\frac2p}(w^{2t})$ for $2/p>1$,
and $(2/p)'=2C^n_m$, we get
\begin{align*}
\Sigma_4^{m,n} & \leq C( C^n_m)^2 (2/p)' [w]_{C^d_{2t}}^{\frac{1}{2}} [w^{2t}]^{\frac{1}{2}}_{A^d_2}
\|f\|_{2}\big\| |gw^{-t}|^p \big\|^{\frac{1}{p}}_{L^{\frac{2}{p}}(w^{2t})}\\
&\leq C (C^n_m)^3 [w]_{C^d_{2t}}^{\frac{1}{2}}
[w^{2t}]^{\frac{1}{2}}_{A^d_2} \|f\|_{2}\| g\|_{2}.
\end{align*}
The theorem is proved.
\end{proof}

\end{document}